\documentclass[12pt, fleqn]{amsart}

\usepackage{amssymb}
\usepackage{amsthm}
\usepackage{amsmath}
\usepackage{fontenc}
\usepackage{amsaddr}
\usepackage{inputenc}
\usepackage{enumitem}
\usepackage{hyperref}
\usepackage{xcolor}

\usepackage{amsrefs}
\usepackage{mathrsfs}
\usepackage{verbatim}

\usepackage[top = 1in, left = 1.25in, right = 1in, bottom = 1in]{geometry}

\usepackage{tikz}
\usetikzlibrary{scopes,arrows,decorations.pathmorphing,decorations.markings,backgrounds,positioning,fit,petri}
\tikzstyle{line} = [draw, thick, -latex']
\tikzstyle{bigArrow} = [thick, decoration={markings,mark=at position
   1 with {\arrow[semithick]{open triangle 60}}},
   double distance=1.4pt, shorten >= 5.5pt,
   preaction = {decorate},
   postaction = {draw,line width=1.4pt, white,shorten >= 4.5pt}]
\tikzstyle{innerWhite} = [semithick, white,line width=1.4pt, shorten >= 4.5pt]
\definecolor{mycolor}{RGB}{0,160,0}

\def\real{\hbox{\rm\setbox1=\hbox{I}\copy1\kern-.45\wd1 R}}
\def\natural{\hbox{\rm\setbox1=\hbox{I}\copy1\kern-.45\wd1 N}}

\newcommand{\be}{\begin{equation}}
\newcommand{\en}{\end{equation}}

\newtheorem{theorem}{Theorem}[section] 
\newtheorem{lemma}[theorem]{Lemma}     

\newtheorem{proposition}[theorem]{Proposition}

\theoremstyle{definition}
\newtheorem{definition}[theorem]{Definition}

\newtheorem*{ack*}{Acknowledgment}

\theoremstyle{remark}

\theoremstyle{case}

\numberwithin{equation}{section}

\title[Symbolic Rank One]{Constructive Symbolic Presentations of {Rank One} Measure-Preserving Systems} 

\author{Terrence Adams}
\address{U.S. Government,
9800 Savage Rd,
Ft. Meade, MD 20755 {USA}}
\email{terry@ganita.org}

\author{S\'ebastien Ferenczi}
\address{IMPA, CNRS UMI 2924,
Estrada Dona Castorina 110, Rio de Janeiro RJ 22460-320, Brasil}
\email{ssferenczi@gmail.com}

\author{Karl Petersen}
\address{Department of Mathematics,
CB 3250 Phillips Hall,
University of North Carolina,
Chapel Hill, NC 27599 USA}
\email{petersen@math.unc.edu}

\date{\today}

\begin{document}
	

\begin{abstract}
	Given a rank one measure-preserving system defined by cutting and stacking with spacers, 
we produce a rank one binary sequence such that its orbit closure under the shift transformation, 
with its unique {nonatomic} invariant probability, is isomorphic to  
the given system. In 
particular, the classical dyadic odometer is presented in terms of a recursive sequence of blocks on the two-symbol alphabet $\{0,1\}$. 
The construction is accomplished using a definition of rank one in the setting of 
adic, or Bratteli-Vershik, systems.
\end{abstract}

\maketitle

\section{Introduction}
 {\em Rank one} measure-preserving systems were introduced in \cite{chacon}, and got their present name in \cite{orw}; 
 they have been extensively studied, including recent developments such as {\cite{ForemanWeiss2016,ForemanWeiss2015,gaohill,Danilenko2015,Danilenko2016}}. The survey 
 \cite{Ferenczi1997} tried to collate the 
 various definitions of rank one, but it has two shortcomings: first, it ignores
 the theory of {\em adic systems} (see for example \cite{verliv}, and, for background, \cite{Durand2010}), 
 though rank one systems have a particularly simple and useful adic presentation 
 (which, to our knowledge, has never appeared in print). Then, the nicest (and most used 
 in recent papers) definition of rank one, called the constructive symbolic definition in the survey, was 
 not known at the time to be equivalent to the others, as the {\em odometers}, which are rank one systems
 by all the other definitions, do not satisfy it (see the discussion in Section 3 below).

The present note aims to fill both of these gaps,
by giving one more definition of rank one systems, the {\em adic definition}, and then by using
it to show that any rank one system is measure-theoretically isomorphic to  a rank one system with an extra  
property we call {\em essential  expansiveness}, which in turn implies that it satisfies 
the constructive symbolic definition.

{Rank one systems, as described in Definition \ref{dgc} below, were modified,  
	in a different context, by the second author in \cite[Theorem 6.3]{fhz3}  and \cite[Theorem 4.8]{fz2}, 
	by replacing, at suitably rare steps $n$, a relatively small number of  $n$-towers
	by stacks of spacers of the same height,
	and this technique may well have been known before. 
	In 2013 the use of this kind of construction to give a
	symbolic constructive model in the particular case of odometers was discussed by the second author with Aaron Hill, but
	nothing was written. 
	In \cite[Remark 2.10]{Danilenko2015} it is stated that each rank one system is isomorphic to one with $a_{n,q_n-1}>0$ for infinitely many $n$. Indeed,
	combining the proof of \cite[Theorem 2.8]{Danilenko2015} with \cite{delJunco1998} and  \cite[Appendix]{Kalikow1984}, one may produce an isomorphism with a constructive symbolic rank one system.
	Here, we show explicitly how to build for each rank one $\mathbb Z$-action an isomorphic system that is essentially expansive, i.e., satisfies the constructive symbolic definition of rank one.  Our new system will satisfy the hypothesis of Proposition \ref{pea}, a stronger condition than the one in \cite[Remark 2.10]{Danilenko2015}.}

{
\begin{ack*}
	This project was advanced through conversations between TA and KP at the April 2016 University of Maryland-Penn State Workshop on Dynamical Systems and Related Topics and between SF and KP at the May 2016 
	Cieplice Conference on New Developments Around the x2 x3 Conjecture and Other Classical Problems in Ergodic Theory.
	We thank the organizers and supporting institutions of those meetings for providing these opportunities for collaboration.  
	{We also thank A. Hill for conversations on this topic and A. Danilenko for bringing our attention to \cite{Danilenko2015} and quoting our main result, with a new proof, in \cite{Danilenko2016}.}
\end{ack*}
}

\section{Geometric and symbolic definitions}
Henceforth we take as definition of a rank one system the constructive geometric
definition of \cite{Ferenczi1997}; note that the property of being rank one is invariant under
measure-theoretic isomorphism; indeed, it is equivalent to building our system  on the unit 
interval  by  cutting and stacking 
in a certain way, and allowing such isomorphisms. By a {\em measure-preserving system} we mean a nonatomic Lebesgue probability space together with an invertible measure-preserving transformation.

\begin{definition}\label{dgc} A measure-preserving system $(X, \mathcal A, T, \mu )$ is {\em rank one} if there exist sequences of
positive integers $q_n, n\in {\mathbb N}\cup \{0\}$, with $q_n>1$ for infinitely many $n$,  and $a_{n,i}, n \in {\mathbb N}\cup \{0\}, 
0\leq i \leq q_n-1$, such that, if $h_n$ and $s_n$ are defined by 
\be\label{eq:h} h_0=1,
h_{n+1}=q_nh_n+\sum_{j=0}^{q_n-1}a_{n,i},
\en 
then 
\be\label{eq:sumcond}
\sum_{n=0}^{\infty}{{h_{n+1}-q_nh_n}\over{h_{n+1}}} =\sum_{n=0}^{\infty}{{\sum_{i=0}^{q_n-1}a_{n,i}}\over{h_{n+1}}} <\infty
\quad\text{ (equivalently} \sum_{n=0}^\infty \frac{\sum_{i=0}^{q_n-1}a_{n,i}}{q_nh_n} < \infty ),
\en
and subsets of $X$, denoted by $F_n, n \in {\mathbb N}\cup \{0\}$, $F_{n,i}, n
\in\mathbb{N}, 0 \leq i \leq q_n-1$, and  $S_{n,i,j}, n \in\mathbb{N}, 0 \leq i
\leq q_n-1, 0 \leq j \leq a_{n,i}-1$, (if $a_{n,i}=0$
there are no $S_{n,i,j}$), such that for all $n$ \begin{itemize} \item
$(F_{n,i}, 0 \leq i \leq q_n-1)$ is a partition of $F_n$, \item the
$T^kF_n, 0 \leq k \leq h_n-1$, are disjoint, \item $
T^{h_n}F_{n,i}=S_{n,i,1}$ if $a_{n,i} \neq 0$, \item
$T^{h_n}F_{n,i}=F_{n,i+1}$ if $a_{n,i}=0$ and $i<q_n-1$, \item
$TS_{n,i,j}=S_{n,i,j+1}$ if $j<a_{n,i}-1$, \item
$TS_{n,i,a_{n,i}-1}=F_{n,i+1}$ if $i<q_n-1$, \item $F_{n+1}=F_{n,0}$,
\end{itemize} and the partitions $P_n=\{F_n, TF_n, ..., T^{h_n-1}F_n,
X\setminus\cup_{k=0}^{h_n-1}T^kF_n\}$ are increasing to $\mathcal{A}$.
\end{definition}

The parameters $q_n$ and $a_{n,i}$ determine the transformation $T$ up to 
a set of measure $0$, but of course for a given $T$ several
sets of parameters are possible (this is the point of the whole paper).
The sets $T^kF_n, 0 \leq k \leq h_n-1,$ are called the {\em levels} of
the $n$-{\em towers} (or stacks). The sets $S_{n,i,j},  0 \leq i
\leq q_n-1, 0 \leq j \leq a_{n,i}-1$ are the $n$-{\em spacer} levels.

Note that if we had $q_n=1$ ultimately,
then because of the condition on $h_n$ we should have also $a_{n,0}=0$ for all $n$
large enough, and then $T$ would not be defined almost everywhere. 

This construction determines naturally an associated constructive symbolic definition.

\begin{definition}\label{dsc}
A system is {\em constructive symbolic rank one} if it consists of 
the shift on the orbit closure of the sequence $u$ defined as follows, with its unique
nonatomic invariant measure. 
	For each $n=0,1,\dots$ the initial $n$-block of $u$ is $B_n$. 
	The blocks $B_n$ are defined 
 recursively by concatenation: 
\be 
B_0=0,
B_{n+1}=B_n 1^{a_{n,0}}B_n \dots B_n 1^{a_{n,q_n-1}} \text{ for } n \geq 0.
\en 
The parameters $a_{n,i}$ are as in Definition \ref{dgc} and are assumed to satisfy (\ref{eq:h}) and (\ref{eq:sumcond}).
\end{definition}

The symbols $1$, corresponding to spacer levels, are called spacers.

\section{Odometers}
The  classical dyadic odometer, also called the von Neumann-Kakutani 
adding machine, can be defined by  the constructive geometric definition as a rank one system with $q_n=2$ for all $n$, and  $a_{n,i}=0$ for all $n$ and $i$; but then the blocks $B_n$ on the alphabet $\{0,1\}$ defined recursively by $B_0=0, B_{n+1} =B_nB_n$ lead to  a trivial one-point symbolic dynamical system, while the phase space of the dyadic odometer is the unit interval. Thus this set of parameters $q_n$ and $a_{n,i}$ do not provide a symbolic constructive definition of the  odometer; in \cite[p. 15]{Ferenczi1997} and \cite[p. 124]{Fogg} it is stated that the dyadic odometer is not known to be rank one by the constructive symbolic definition.  \\

Note that there exists a symbolic model for the dyadic odometer: namely, the shift on the orbit closure of the fixed point $u$ of the substitution $\sigma$ defined by $0 \to 01, 1 \to 00$ (a Toeplitz sequence, called the period-doubling sequence and constructed by Garcia and Hedlund \cite{GarciaHedlund1948}) is measure-theoretically isomorphic to the dyadic odometer. As an aside, we show we cannot use it to find a  constructive symbolic definition for this rank one system:

\begin{proposition}
	\label{rem:toeplitz}
 No sequence $x$ in the orbit closure of $u$  can be built by blocks $B_n$ as in Definition \ref{dsc}. \end{proposition}
\begin{proof}	Any sequence $x$ in the orbit closure of $u$ is the concatenation of the blocks $0100$ and $0101$. Therefore spaces between appearances of $0100$ in $x$ are multiples of $4$ (there is no appearance of $0100$ across a juncture of $0100$ and $0101$).
	Suppose $x= B \, 1^{s_0}\, B \, 1^{s_1} \, B\, \dots$, where $B$ is some block in
	a supposed
	 constructive symbolic rank one construction of $x$ that is long enough to contain $0100$, and $B$ starts and ends with the symbol $0$. 
	The $s_i$ are bounded because the sequence has minimal orbit closure. The sequence $x$ does not contain the block $11$, so each $s_i$ is $0$ or $1$. If both $0$ and $1$ appear among the $s_i$, then we can find in $x$ a stretch where the spaces between appearances of $B$, and hence between appearances of $0100$, are not all multiples of $4$. 
	Therefore either  $s_i=0$ for all $i$ or  $s_i=1$ for all $i$, and in both cases $x$ is periodic.\end{proof}

The general odometers are rank one systems with $a_{n,i}=0$ for all $n$ and $i$.

\section{Adic definition}\label{sad}
We give now a definition of rank one which was not known at the time of  \cite{Ferenczi1997}: it  is   a translation of the constructive geometric definition into the language of adic transformations defined on paths in 
Bratteli diagrams. 

\begin{definition}\label{da} A system is {\em rank one} if it is measure-theoretically isomorphic to the  following system, where the $q_n$ and $a_{n,i}$ are as in Definition \ref{dgc}
	and satisfy the sum condition (\ref{eq:sumcond}).
	
{L}et $V = \{ v_{-1}, v_{0,0}, v_{0,1},\ldots,  v_{n,0}, v_{n,1}, \ldots \}$ 
represent an infinite set of nodes (or vertices).   
The integer $n$ is called the {\em level} of  the vertex $v_{n,i}, i=0,1$. 
For each $n=-1,0,1,\dots$ there are directed edges from level $n$ to level $n+1$.
Let $e_{-1,0}$ be an edge connecting $v_{-1}$ to $v_{0,0}$, 
and $e_{-1,1}$ be an edge connecting $v_{-1}$ to $v_{0,1}$.  
For $n\in \natural \cup \{ 0 \}$, let $e_{n, i}$ for $0\leq i < q_n$ 
be edges connecting $v_{n,0}$ to $v_{n+1,0}$.  
Let $e_{n,i,j}$ for $0\leq i < q_n$ and $0\leq j <a_{n,i}$ 
be edges connecting $v_{n,1}$ to $v_{n+1,0}$.  
Let $d_n$ be a single edge connecting $v_{n,1}$ to $v_{n+1,1}$. 
Order the edges entering $v_{n+1,0}$ as 
\[
\ldots \prec e_{n,i} \prec e_{n,i,0} \prec e_{n,i,1} \prec \ldots \prec e_{n,i,a_{n,i}-1} \prec e_{n,i+1} \prec \ldots . 
\]
Let $E$ be the set of all edges, and let $G = \{ V, E \}$ be the infinite directed acyclic graph 
with the ordering defined above.  Let $X_0$ be the set of infinite paths beginning 
at the root node $v_{-1}$. 
The ordering of edges arriving at the nodes  defines a  lexicographic ordering of infinite paths that start at $u_{-1}$: we say that $x<y$ if 
there is an $N$ such that
$x_n=y_n$ for all $n>N$ and $x_{N}<y_{N}$.
	If $X_{\max}$ denotes the set of maximal paths, then 
the  adic transformation $T$ on $X=X_0 \setminus X_{\max}$ is defined to be the map that assigns to each path its successor in this ordering. 

A {\em cylinder set} is determined by a finite set $E_0$ of edges forming a path from the root vertex $v_{-1}$ to a vertex $v$: it is the set of all the infinite paths in $X$ that pass through all the edges in $E_0$.
Let $\mathcal{A}$ be the $\sigma$-algebra generated by all cylinder sets.  
  There is a unique probability measure that for each $v \in V$ assigns the same mass 
to all cylinder sets determined by paths from $v_{-1}$ to $v$.  
The adic transformation $T$ preserves this measure. 
Denote by $P_k'$ the partition of $X$ into the cylinder sets determined by paths from $v_{-1}$ to vertices at level $k$. 

\end{definition}

We show how to translate the constructive geometric definition into the adic definition: for  almost all $x$ in the geometric system, either there is a unique $M=M(x)\geq 0$ such that $x$ is in $S_{M,i,j}$, for 
some  $0 \leq i=i(x)
\leq q_n-1, 0 \leq j=j(x) \leq a_{n,i}-1$, or else $x$ is in $F_0$, in which case we take $M(x)=-1$. Then, for all $n>M$, $x$ is in  some $T^kF_{n,i}$ for some  $0 \leq i=i(n,x)
\leq q_n-1$, which means that $x$ 
is in the $i$-th (from the bottom) copy of the $n$-tower inside the $n+1$-tower.

 Then in the adic system $x$ becomes the path (still denoted by $x$)
\begin{eqnarray}
 \label{eq:ad} \qquad
\left\{ \begin{array}{ll}
(e_{-1,1},d_0,\ldots,d_{M-1},e_{M,i(x),j(x)}
e_{M+1,i(M+1,x)}, e_{M+2,i(M+2,x)},\ldots)&\mbox{
if }M>0,\\
(e_{-1,1},e_{0,i(x),j(x)},e_{1,i(1,x)}, e_{2,i(2,x)},\ldots)&\mbox{ if }
M=0,\\
 (e_{-1,0},e_{0,i(0,x)},e_{1,i(1,x)}, e_{2,i(2,x)},\ldots)&\mbox{  if }M=-1.\\
\end{array} \right. \end{eqnarray}

On the other hand, in the adic system, for each $n \in \natural \cup \{ 0 \}$, define the  cylinder set 
\be
C_{n} = \{ x : x(m) = e_{m,0,0}, m< n \} . 
\en 
 Then the  $h_n$ of Definition \ref{dgc} is the number of paths from the root node $v_{-1}$ 
to the node $v_{n, 0}$, often called the {\em dimension} of the vertex $v_{n, 0}$, and the  basis $F_n$ of the tower in Definition \ref{dgc} is the cylinder $C_n$. 
For $n=0$, 
the partition $P_0'$ is $\{C_0,D_1\}$, where $D_1=C_0^c$ is the set of
{\em spacer paths} that go through the vertex $u_{0,1}$, while $C_0$ is the set 
of {\em nonspacer paths} that do not.\\

What determines a point $x$ in a rank one system is the sequence $J_n$ of the numbers of the level where $x$ lies in the $n$-tower. The adic definition
allows us to compute them explicitely from the path $x$. 

\begin{lemma}\label{mad} Let $x$ be defined by (\ref{eq:ad}) with $M(x)=M$; then for all $n>M$, $x$ is in $T^{J_n(x)}C_n$ where
\begin{eqnarray}\quad\left\{\begin{array}{ll}
\mbox{ if }M=-1, &J_0(x)=0,\\
\mbox{ if }M\geq 0,&
 J_{M+1}(x)=(i(x)+1)h_{M}+a_{M,0}+\ldots+a_{M,i(x)-1}+j(x),\\
\mbox{ if }n\geq M+1,&J_{n+1}(x)=i(n,x)h_n+a_{n,0}+\ldots+a_{n,i(n,x)-1}+J_n(x). \end{array} \right.
\end{eqnarray}
\end{lemma}
\begin{proof} This 
is a direct consequence of the construction of the towers in Definition \ref{dgc} and its adic translation above.
\end{proof}

  \section{Essential expansiveness}
We look now at cases where the constructive symbolic definition of rank one is equivalent to the other definitions. For this, we use a notion of expansiveness which comes from the general theory of adic
and topological dynamical systems, see for example \cite{dm2008}.

\begin{definition} For $k \in {\mathbb N}\cup \{0\}$, a rank one system defined
by Definition \ref{dgc}, resp. Definition \ref{da},  is {\em essentially $k$-expansive} if  the partition $P_k$ of $X$,
	resp. $P_k'$,
	  generates the full $\sigma$-algebra  under the  transformation $T$.   
	  \end{definition}

When a rank one system defined by Definition \ref{da} is essentially $0$-expansive, the symbolic system 
that results from coding the orbits of its paths according to their initial edges and taking the closure 
 coincides with the symbolic system 
defined in Definition \ref{dsc}. With the same parameters in each case, it is measure-theoretically isomorphic to both 
 the adic system $(X,T)$ of Definition \ref{da} and a measure-preserving system satisfying
 Definition \ref{dgc}. \\
 
 A necessary and sufficient condition for essential $0$-expansiveness of rank one systems, due to Kalikow 
 \cite[Appendix]{Kalikow1984}, is that the sequence $u$ in  Definition \ref{dsc} is not 
 periodic. {This in turn is realized for example when the parameters in 
  Definition \ref{dgc}  or \ref{da} satisfy
\be\label{kalspa}
\sup_{0\leq m\leq n, 0\leq i\leq q_{n+1}-1} a_{m,q_m-1}+a_{m+1,q_{m+1}-1}+\ldots a_{n,q_n-1}+a_{n+1,i}=\infty,
\en
  as this implies $u$ has unbounded strings of $1$, while it has symbols 
$0$ infinitely often because $q_n$ is at least $2$ for infinitely many $n$.} \\

Here, in order to make the present paper more self-contained, we introduce a stronger sufficient condition,
allowing an elementary proof of essential expansiveness.

\begin{proposition}\label{pea} 
	A rank one system defined by Definition \ref{dgc}  or \ref{da}
is essentially $0$-expansive if
$a_{n,q_n-1}>a_{n,i}$ for all $n$ and $i<q_n-1$.
\end{proposition}
\begin{proof}
Let  $\mathcal D$ the smallest $S$-invariant $\sigma$-algebra that contains $C_0$ and $D_1$: we shall 
show that it contains the cylinder set
$C_n$ of Section 4 for all $n$, thus all cylinder sets, 
and therefore equals the full $\sigma$-algebra of $Y$.

Indeed, $C_0$ is in $\mathcal D$. Suppose that $n \geq 0$ and
 $C_n$ is in $\mathcal D$.
 
Let $W_0 = T^{h_n}C_n \cap D_1$.  
For $i$ such that $0 \leq i < a_{n,q_n-1}$, define $W_{i+1} = TW_i \cap D_1$.  Then, by our hypothesis, 
 $W_{a_{n,q_n-1}}$ is exactly $T^{h_{n+1}-1}C_{n+1}$. 
 As all the $W_i$ are in $\mathcal D$, so
 is $C_{n+1}$, which completes our induction. 
\end{proof}

Note that our proposition gives a sufficient condition for essential expansiveness, but
it is not necessary: the so-called Chacon map given by Definition \ref{dsc} with the recursion 
rule $B_{n+1}=B_nB_n1B_n$ does not satisfy it, but, as $u$ is not periodic, the system
defined by Definition \ref{dgc} with the same parameters  is essentially expansive.

Note that if  $q_n=2$ and there is a single spacer, essential expansivity implies that 
$a_{n,q_n-1}=a_{n,1}>a_{n,0}$ 
at least for infinitely many $n$;
otherwise, 
 the recursion rule is  $B_{n+1}=B_n1B_n$ ultimately, and $u$ is periodic while
 the corresponding system in Definition \ref{dgc} is an odometer. 
Indeed, it is proved in \cite[Lemma 10]{ElAbdaloui2014} that if the sequence
$u$ defined as in Definition \ref{dsc} is periodic, then the corresponding rank one                                                                                                                                                                                                                                                                                                                                                                                                                                                                                                                                                                                                                                                                                                                                                                                                                                                                                                                                                                                                                                                                                                                                           
system defined as in Definition \ref{dgc} is isomorphic to an odometer.

\section{Essentially expansive construction}                   
{We turn now to our main {objective}, the construction of an essentially expansive model                      
(and thus of a constructive symbolic model) for every rank one transformation. 
 We use the adic definition, which is well adapted to the.building of isomorphisms.}

\begin{theorem} Every rank one system is measure-theoretically isomorphic to an essentially $0$-expansive rank one system.\end{theorem}
The remainder of this section is devoted to the proof of the theorem.

\subsection{Definition of the new system} 
We start from the rank one system $(X,T)$ in Definition \ref{dgc} or \ref{da}. We choose a 
strictly increasing sequence $m_n$ with $m_0=0$ and such that if $H_n=h_{m_n}$ then
\be
\sum_{n=0}^{\infty} \frac{H_n}{H_{n+1}} < \infty. 
\en

We shall now look at the same system $(X,T)$, but we build it by considering only the steps $m_n$. Thus the parameters $q_n$ and $a_{n,i}$ are replaced by new parameters $Q_n$ and $A_{n,i}$, which we compute now.\\

Namely, we have 
\be Q_{n} = \prod_{i=m_n}^{m_{n+1}-1} q_i.\en 

We compute $A_{n,i}$ for $0\leq i\leq Q_n-1$. To do that, we write in a unique way
\be i=g_0+\sum_{j=1}^{m_{n+1}-m_n-1}g_j\prod_{k=0}^{j-1}q_{m_n+k}  \en
with $0\leq g_j \leq q_{m_n+j}-1$.
Then let $l$ be such that $g_k=q_{m_n+k}-1$ for all $k<l$ and 
$g_l\neq q_{m_n+l}-1$,  or $l=m_{n+1}-m_n-1$ if this last inequality is never satisfied; then we have
\be A_{n,i}=\sum_{k=0}^la_{m_n+k,g_k}.\en

Thus, the $H_n$ are indeed the heights of the stacks and the $A_{n,i}$ the numbers of spacers for the system $(X,T)$ 
after collapsing the stages between $m_n+1$ and $m_{n+1}-1$. 
As $(X,T)$ preserves a finite measure, the equivalent of condition (\ref{eq:sumcond}) still applies, thus

\be\label{eq:sumcon2}
\sum_{n=0}^\infty \frac{\sum_{i=0}^{Q_n-1}A_{n,i}}{H_{n+1}}<+\infty.
\en
We consider the rank one system $(X,T)$, as defined by the parameters $Q_n$ and $A_{n,i}$, with Definition \ref{da}; its edges are denoted as in that definition, by 
$d_n$, $e_{n,i}$ and $e_{n,i,j}$.\\

Let $\overline{A}_n = \max{ \{ A_{n,i} : 0\leq i\leq Q_n-1\} }$.  Choose $0\leq i_n\leq Q_n-1$ to be the last $i$ 
such that
$A_{n,i}+H_n+A_{n,i+1} +H_n+\ldots A_{n, Q_n-1}>\overline{A}_n$,  Namely

\be
\overline{H}_n = (Q_n-i_n-1)H_n +
\sum_{i\geq i_n}  A_{n,i} > \overline{A}_n, 
\en
but 
\be
  (Q_n-i_n-2)H_n+
\sum_{i>i_n} A_{n,i} \leq \overline{A}_n . 
\en

We define  now a new {rank one}  system $(Y,S)$ by modifying the system $(X,T)$ (now built with the towers at stages $m_n$)
in the following way: in the building of the tower at stage $m_{n+1}$ by  subsets of the tower at stage $m_n$ and spacers, we replace the upper
$(Q_n-i_n-1)$ of these sub-towers by spacers. The index $i_n$ has {been} chosen such that this change is the smallest one allowing the new system to satisfy Proposition \ref{pea}, and we shall show below that the change is small enough, first to allow the new system to preserve a finite measure, and then to allow it to be isomorphic to the initial system.

 We define now this system precisely, using the adic formalism.

  The set of vertices of our Bratteli diagram is 
$\{ u_{-1}, u_{i,0}, u_{i, 1} : i\in \natural \cup {0} \}$. 

Let $r_{-1,0}$ be a {directed} edge {from} $u_{-1}$ to $u_{0,0}$, and 
$r_{-1,1}$ {from} $u_{-1}$ to $u_{0,1}$. 
Let $d_n$ be a single edge connecting $u_{n,1}$ to $u_{n+1,1}$. For $i \leq i_n$,
 let $r_{n,i}$ be  edges from  $u_{n, 0}$ to $u_{n+1,0}$;  
for  $i>i_n$, 
there is no edge. 
$r_{n,i}$;
thus there are 
$i_n+1$ 
edges $r_{n,i}$.

Let 
$s_{n, i, j}$ be edges connecting $u_{n,1}$ to $u_{n+1,0}$ 
for $0 \leq j < A_{n,i}$, 
and $i<i_n$. 
Let $s_{n, i_n, j}$ be edges connecting $u_{n,1}$ to $u_{n+1,0}$ 
for $0 \leq j < \overline{H}_n$.

Order the edges entering $u_{n+1, 0}$ in the following manner (with  obvious modifications if $i_n$ is 
the last element of $U_n$): 
\begin{itemize}
\item $r_{n, i} \prec s_{n, i, 0} \prec s_{n, i, 1} \prec 
\ldots \prec s_{n, i, A_{n, i}-1}\prec r_{n,i+1}$  
for $i<i_n$; 
\item  $r_{n, i_n} \prec s_{n, i_n, 0} \prec s_{n, i_n, 1} \prec \ldots \prec
 s_{n, i_n, \overline{H}_n - 1}$.
\end{itemize}

This ordering on the Bratteli diagram defines an invertible transformation $S$  
on the set of all non-maximal paths $Y$. This new system is {rank one}, and we get its parameters 
by  replacing $Q_n$ by $Q'_n=i_n+1$, $A_{n,i}$ by $A'_{n,i}=A_{n,i}$
if $i<i_n$, $A'_{n,i_n}=\bar H_n$, the $A'_{n,i}$ being defined for $0\leq i <Q'_n$. The heights of the stacks remain the same, $H'_n=H_n$.
As in Section \ref{sad}, we define quantities $M'(x)$ and $J'_n(x)$.

We have 

\be\begin{aligned}
\sum_{i=0}^{Q'_n-1} A'_{n,i}&=\sum_{i=0}^{Q_n-1} A_{n,i}+
 (Q_n-i_n-1)H_n\\
 &\leq \sum_{i=0}^{Q_n-1}A_{n,i}+ \bar A_n +H_n \leq 2\sum_{i=0}^{Q_n-1}A_{n,i}+H_n.
\end{aligned}
\en

 Since $\sum_{n=0}^{\infty} {H_n} / {H_{n+1}} < \infty$ by construction, condition (\ref{eq:sumcon2}) implies that 
  {
  	\be
  	\sum_{n=0}^{+\infty} \frac{\sum_{i=0}^{Q'_n-1}A'_{n,i}}{H'_{n+1}}<+\infty.
  	\en}
  Thus
there is a unique invariant nonatomic probability 
measure for the {rank one} transformation $S$ which assigns the same 
measure to all cylinder sets ending at a given vertex in the left column.  

$S$ is essentially $0$-expansive by Proposition \ref{pea}, as its last stacks of spacers are 
of height $\overline{H}_n$ while the others are of height at most $\overline{A}_n$
for all $n$. Now we  show that $S$ is isomorphic to $T$.  

\subsection{Definition of the isomorphism} 
\label{iso}

Let $E_n$
denote the set of $x \in X$  which are either in the $n$-spacers or in the top $\overline H_n$ levels of the $n+1$-tower; namely $E_n$ is the set of $x$ such that
\begin{itemize}
\item either $x(n)=e_{n,i,j}$ for some $i,j$;
\item or $x(n)=e_{n,i}$ for some $i>i_n$.
\end{itemize}

Let 
\be
E = \bigcap_{k=0}^{\infty} \bigcup_{n=k}^{\infty} E_n ; 
\en
\be
\overline E = \bigcup_{-\infty}^{+\infty}T^kE. 
\en
Now $E_n$ is made with exactly $\sum_{i=0}^{Q_n-1} A_{n,i}+(Q_n-i_n-1)H_n$ iterates of $C_{n+1}$, and we have
$(Q_n-i_n-2)H_n\leq \overline{A}_n$, thus 
\be \mu (E_n)\leq \frac{\sum_{i=0}^{Q_n-1} A_{n,i}+\overline{A}_n+H_n}{H_{n+1}} \en
Because of (\ref{eq:sumcon2}) and $\sum_{n=0}^{\infty} {H_n} / {H_{n+1}} < \infty$ we have $\sum_{n=0}^{\infty} \mu (E_n) < \infty$ and 
\be \mu (E) =\mu(\overline E)= 0.\en
 
If $x \notin \bigcup_{n=0}^{\infty} E_n$ (i.e., $x$ is in no ${E_n}$), 
define $N(x) = -1$; 
otherwise, for $x \notin E$, define 
\be
N(x) = \max{ \{ n : x \in E_n \} } . 
\en
Note that, by definition of the $E_n$, $N(x)\geq M(x)$ (of Section \ref{sad}). Hence $J_{N(x)+1}(x)$ exists and can be computed from $x$ by Lemma \ref{mad}.

Now we define our isomorphism $\phi : X \to Y$. 
For $x \in X \setminus E$,  define $y = \phi (x)$ as follows: if $N(x) = -1$, define $y(-1) = r_{-1,0}$ and 
\be
y(n) = r_{n,i} \mbox{ whenever } x(n) = e_{n,i} .
 \en
Otherwise, if $0 \leq N(x) =N< \infty$, define $y(-1) = r_{-1, 1}$ and  
\begin{eqnarray} \label{eq:iso}\qquad y_n=
\left\{\begin{array}{lll}
d_n  &&\mbox{ for }\ 0 < n < N \\ 
s_{N,i,j}&\mbox{ if } x(N)=e_{N,i,j}, i \leq i_N&\mbox{ for } n = N \\ 
s_{N,i_N,J_{N+1}(x)-H_{N+1}+\overline H_N}&\mbox{ if } x(N)=e_{N,i},i>i_N&\mbox{ for } n = N \\ 
s_{N,i_N,J_{N+1}(x)-H_{N+1}+\overline H_N}&\mbox{ if } x(N)=e_{N,i,j},i>i_N&\mbox{ for } n = N \\ 
r_{n,i} &\mbox{ if } x(n)=e_{n,i}& \mbox{ for }\ n > N. 
\end{array}\right.
\end{eqnarray}

\subsection{Proof of the isomorphism}

Let $N=N(x)$. From the definition of $\phi$, $\phi(x)$ is never in the $n$-spacers if $N=-1$, otherwise the last $n$ for which $\phi(x)$ is in the $n$-spacers is $N$. Thus
\be M'(\phi(x))=N(x). \en
We compute $J'_{N+1}(\phi (x))$: if $x(N)=e_{N,i,j}$, then, by definition of $\phi$, $J'_{N+1}(\phi (x))=J_N(x)$; in the other cases, $\phi(x)(N)=
s_{N,i_N,J_{N+1}(x)-H_{N+1}+\overline H_N}$, which implies that $\phi(x)$ is at distance $\overline H_N-1-(J_{N+1}(x)-H_{N+1}+\overline H_N)=
H_{N+1}-1-J_{N+1}(x)$  from the top of the $N+1$-tower, while, by definition of $J'_{N+1}$, $\phi(x)$ is at distance $H'_{N+1}-1-J'_{N+1}(\phi(x))$ from that top. 
So because $H'_{N+1}=H_{N+1}$ we get
\be J'_{N+1}(\phi(x))=J_{N+1}(x). \en
For $n>N+1$, neither $x$ nor $\phi(x)$ can be in the top $\overline H_{n-1}$ levels in their respective $n$-towers. Then, by the definition of $\phi(x)(n)$ for $n\geq N+1$ and Lemma \ref{mad}, we get
\be J'_n(\phi(x)) = J_n(x) \mbox{ for all } n>N(x)=M'(\phi(x)). \en

Now, $\phi(x_1)=\phi(x_2)$ implies $J'_n(\phi(x_1))=J'_n(\phi(x_2))$ for all $n$  larger than $M'(\phi(x_1))\vee M'(\phi(x_2))$. Thus $J_n(x_1)=J_n(x_2)$ for all $n$ large enough, and by the rank one property this implies $x_1=x_2$. Thus $\phi$ is injective on $X\setminus E$.\\

Let $E'_n$
denote the set of $y \in Y$  which are in the $n$-spacers, $E' =\bigcap_{k=0}^{\infty} \bigcup_{n=k}^{\infty} E'_n$. For $y$ in $Y\setminus E'$,
 $M'=M'(y)$ is finite. Thus $y(M')=s_{M',i,j}$, and we can compute $J'_{M'+1}(y)$ by Lemma \ref{mad}. Then $y=\phi(x)$ for the (unique by the injectivity proved above) $x$ such that $J_{M'+1}(x)=J'_{M'+1}(y)$, and $x(n)=e_{n;i}$ whenever $y(n)=r_{n,i}$ for all $n>M'$. This $x$ satisfies $N(x)=M'$ by construction. Thus 
 \be \phi(X\setminus E)=Y\setminus E'.\en
 
 Finally, suppose $x$ is in $\overline E$, so that both $x$ and $Tx$ are in $E$. Then, for $n>(N(x)\vee N(Tx))+1$, both $x$ and $\phi(x)$ are in the $n$-tower and not in the topmost level (indeed not in the topmost $\overline H_{n-1}$ levels). Thus $J_n(Tx)=J_n(x)+1$, $J'_n(S\phi(x))=J'_n(\phi(x))+1$, and 
 \be J'_n(S\phi(x))=J_n(Tx)=J'_n(\phi(Tx)),\en 
 for all $n$ large enough. Thus $\phi$ is equivariant on $X\setminus\overline E$, which completes the proof.

\subsection{Variant}
A variation on the same method would be to replace only one tower by spacers: for 
each $n$ we choose one $i_n$
and throughout 
 the above section we replace ``all $i>i_n$" by $i=i_n$. Then the 
 isomorphism works in the same way, but essential expansiveness is proved using the machinery of 
 Kalikow's result and condition (\ref{kalspa}), as the new system has strings of at least $H_n$ 
 spacers.

\section{Back to the odometer} 
 We can apply this construction to the dyadic odometer. We have $q_n=2$, $a_{n,i}=0$ for all $n$ and $i$, $h_n=2^n$, and choose $ m_n= {n(n+1)}/{2}$, 
 so that $H_{n+1}=2^{n+1}H_n$. 
 Then the odometer is built with $Q_n=2^{n+1}$ and $A_{n,i}=0$ for all $n$ and $i$. This implies $\overline{A}_n=0$, 
 and  $i_n=Q_n-2${.} We build the new system by replacing only the top sub-tower with spacers. 
 {Then the new system admits a constructive symbolic definition, 
 with recursion formula
$$B_{n+1}=B_n^{2^{n+1}-1}1^{2^{{n(n+1)}/{2}}},$$ 
	with corresponding constructive geometric {and} adic definitions.
	The following picture gives the first stages of the adic construction, corresponding to $B_1=B_01$ and $B_2=B_1B_1B_111$.}

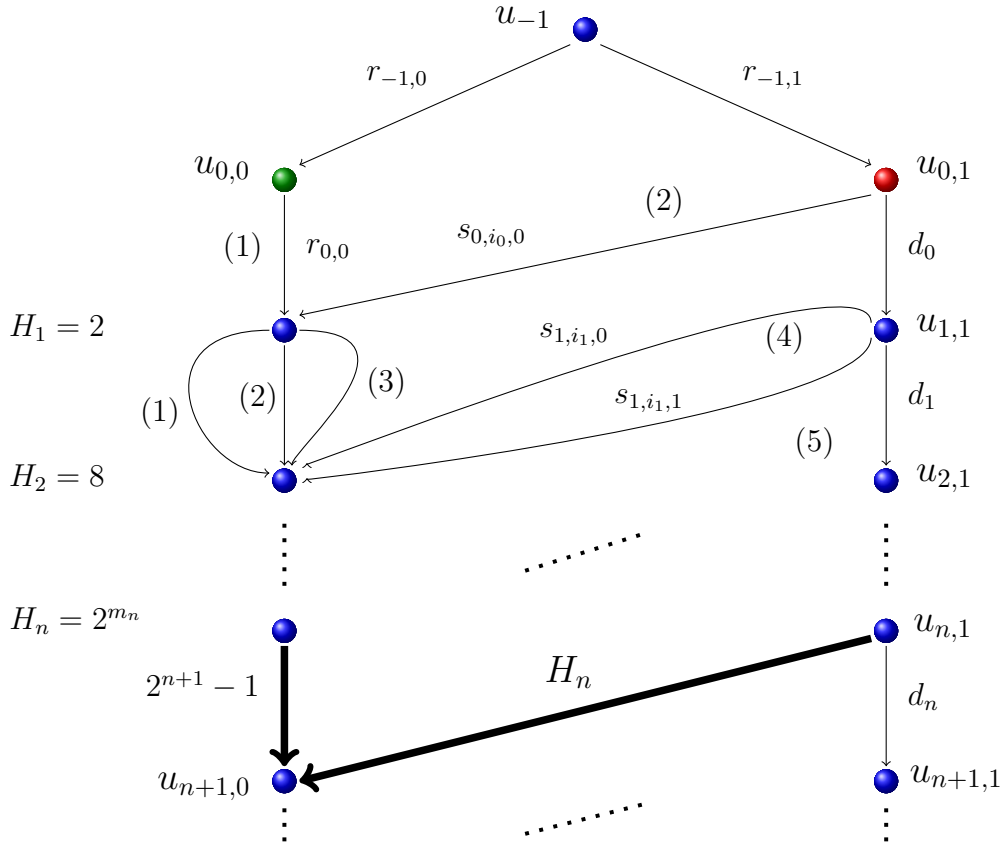
\begin{figure}[!h]
\begin{tikzpicture}

\shade [ball color=blue] (1+6,10) circle (.16) node[](residual0){}; 
\node[black,below left=-0.6cm and 0.0cm of residual0]{ {\large $u_{-1}$ } };

\draw[<-] (1+2.2,8.2)--(1+5.8,9.8) node[blue](myarrow01){};
\node[black,below left=0.0cm and 1.6cm of myarrow01]{ $r_{-1,0}$ }; 

\draw[<-] (1+9.8,8.2)--(1+6.2,9.8) node[blue](myarrow1){};
\node[black,below right=0.0cm and 1.6cm of myarrow1]{ $r_{-1,1}$ }; 

\shade [ball color=mycolor] (1+2, 8) circle (.16) node[](residual1){};
\node[black,below left=-0.6cm and 0.0cm of residual1]{ {\large $u_{0, 0}$ } };

\shade [ball color=red] (1+10,8) circle (.16) node[](residual9){};
\node[black,below left=-0.6cm and -1.6cm of residual9]{ {\large $u_{0, 1}$ } };
\draw[<-] (1+10.0,6.2)--(1+10.0,7.8) node[blue](myarrow5){};
\node[black,below right=0.2cm and 0.0cm of myarrow5]{ $d_0$ }; 

\draw[<-] (1+2.2,6.2)--(1+9.8,7.8) node[blue](myarrow100){};
\node[black,below left=0.1cm and 4.3cm of myarrow100]{ $s_{0,i_0,0}$ }; 
\node[black,above left=-0.6cm and 2.2cm of myarrow100]{ (2) }; 

\draw[<-] (1+2.0,6.2)--(1+2.0,7.8) node[blue](myarrow102){};
\node[black,below left=0.2cm and .0cm of myarrow102]{(1)};
\node[black,below right=0.3cm and 0.0cm of myarrow102]{ $r_{0,0}$ }; 

\shade [ball color=blue] (1+2,6.0) circle (.16) node[](residual2){};
\node[black,below left=-0.5cm and 2.1cm of residual2]{ $H_1=2$ }; 
\draw[<-] (1+2.0,4.2)--(1+2.0,5.8) node[blue](myarrow103){};
\draw[<-] (1+1.8,4.1) .. controls +(left:1.0cm) and +(left:1.8cm) .. (1+1.8,6.0) node[](myarrow104){};
\draw[<-] (1+2.1,4.2) .. controls +(north:0.2cm) and +(right:1.8cm) .. (1+2.2,6.0) node[](myarrow105){};
\node[black,below left=0.2cm and -0.2cm of myarrow103]{(2)};
\node[black,below left=0.6cm and 0.9cm of myarrow104]{(1)};
\node[black,below right=0.2cm and 0.6cm of myarrow105]{(3)};
\shade [ball color=blue] (1+2,4.0) circle (.16) node[](residual3){};
\node[black,below left=-0.5cm and 2.1cm of residual3]{ $H_2=8$ }; 

\shade [ball color=blue] (1+2,2.0) circle (.16) node[](residual4){};
\node[black,below left=-0.6cm and 1.6cm of residual4]{ $H_n=2^{m_n}$ }; 
\shade [ball color=blue] (1+2,0.0) circle (.16) node[](residual5){};
\node[black,below left=-0.4cm and 0.0cm of residual5]{ {\large $u_{n+1, 0}$ } };

\shade [ball color=blue] (1+10,6.0) circle (.16) node[](residual10){}; 
\node[black,below left=-0.5cm and -1.6cm of residual10]{ {\large $u_{1, 1}$ } };
\draw[<-] (1+10.0,4.2)--(1+10.0,5.8) node[blue](myarrow6){}; 
\node[black,below right=0.2cm and 0.0cm of myarrow6]{ $d_1$ }; 

\draw[<-] (1+2.3,4.2).. controls +(north:0.0cm) and +(north:1.0cm) .. (1+9.8,6.1) node[blue](myarrow101){};
\node[black,below left=-0.3cm and 0.6cm of myarrow101]{(4)}; 
\node[black,above left=-0.6cm and 3.2cm of myarrow101]{ $s_{1,i_1,0}$ }; 

\draw[<-] (1+2.3,4.0).. controls +(south:0.0cm) and +(south:1.0cm) .. (1+9.8,5.9) node[blue](myarrow110){};
\node[black,below left=0.9cm and 0.2cm of myarrow110]{(5)};
\node[black,below left=0.4cm and 2.2cm of myarrow110]{ $s_{1,i_1,1}$ }; 

\shade [ball color=blue] (1+10,4.0) circle (.16) node[](residual11){};
\node[black,below left=-0.5cm and -1.6cm of residual11]{ {\large $u_{2, 1}$ } };
\shade [ball color=blue] (1+10,2.0) circle (.16) node[](residual12){};
\node[black,below left=-0.5cm and -1.6cm of residual12]{ {\large $u_{n, 1}$ } };
\draw[<-] (1+10.0,0.2)--(1+10.0,1.8) node[blue](myarrow3){};
\node[black,below right=0.2cm and 0.0cm of myarrow3]{ $d_n$ }; 
\shade [ball color=blue] (1+10,0.0) circle (.16) node[](residual6){}; 
\node[black,below left=-0.5cm and -2.0cm of residual6]{ {\large $u_{n+1, 1}$ } };

\draw[<-, line width=1mm] (1+2.2,0.0) -- (1+9.8,1.9) node[blue](myarrow111){};
\node[black,below left=-0.1cm and 3.2cm of myarrow111]{ {\large $H_{n}$ } };

\draw[<-, line width=1mm] (1+2.0,0.2)--(1+2.0,1.8) node[blue](myarrow108){};
\node[black,below left=-0.0cm and .0cm of myarrow108]{ $2^{n+1} - 1$ };

{[line width=1.5pt]
{[black]
\draw[loosely dotted] (1+2, 2.6) -- (1+2, 3.5);
\draw[loosely dotted] (1+10, 2.6) -- (1+10, 3.5);
\draw[loosely dotted] (1+5.2, 2.8) -- (1+6.8, 3.3);

\draw[loosely dotted] (1+2,-.8) -- (1+2,-.3);
\draw[loosely dotted] (1+5.2,-.7) -- (1+6.8,-.3);
\draw[loosely dotted] (1+10,-.8) -- (1+10,-.3);
}}

\end{tikzpicture}
\caption{Constructive odometer}
\end{figure}

\end{document}